\documentclass[11pt]{amsart} 
\usepackage{amsmath}
\usepackage{amssymb,mathtools,amsthm,mathrsfs}
\usepackage{tikz}
\usepackage{paralist}
\usepackage[misc]{ifsym}
\usepackage{epsfig} 
\usepackage{epstopdf} 
\usepackage[colorlinks=true]{hyperref}
\hypersetup{urlcolor=blue, citecolor=red}
\allowdisplaybreaks

\newtheorem{theorem}{Theorem}[section]

\newtheorem{lemma}[theorem]{Lemma}
\newtheorem{proposition}[theorem]{Proposition}

\theoremstyle{definition}
\newtheorem{definition}[theorem]{Definition}
\newtheorem{remark}[theorem]{Remark}

\newcommand{\Bu}{{\boldsymbol{u}}}
\newcommand{\Be}{{\boldsymbol{e}}}
\newcommand{\Bx}{\boldsymbol{x}}
\def\R{\mathbb{R}}
\def\dd{\mathop{}\!\mathrm{d}}
\def\ii{\mathop{}\!\mathrm{i}}

\title[Convection in multilayer porous media]
{Vanishing permeability limit of convection in multilayer porous media} 

\author[Kaijian Sha and Xiaoming Wang]{}


\thanks{This work is supported by NSFC grant 12271237 }

\thanks{$^*$Corresponding author: Xiaoming Wang}

\begin{document}
\maketitle


\centerline{\scshape
Kaijian Sha$^{{\href{mailto:kjsha11@eitech.edu.cn}{\textrm{\Letter}}}1}$
and Xiaoming Wang$^{{\href{mailto:wxm.math@outlook.com}{\textrm{\Letter}}}*1}$}

\medskip

{\footnotesize
 \centerline{$^1$School of Mathematical Sciences, Eastern Institute of Technology, Ningbo, China}
} 

\medskip


\bigskip

 \centerline{(Communicated by Handling Editor)}


\begin{abstract}
We analyze the asymptotic behavior of the Boussinesq–Darcy system describing convection in layered porous media in the limit where the permeability of one layer tends to zero. We show that the limiting dynamics are governed by the Boussinesq–Darcy model with an impermeable layer, both in terms of convergence of solutions in $L^2$ on finite time intervals and convergence of the corresponding global attractors. This limit is singular, as the pressure equation becomes degenerate when the permeability vanishes in part of the domain, resulting in a loss of uniqueness of the pressure in the impermeable layer. This difficulty is resolved by combining uniform estimates in the permeable layers with refined control of the pressure equation in the vanishing-permeability layer. The results provide a rigorous description of the zero-permeability limit in layered porous-media convection models.

\end{abstract}


\section{Introduction}
 
Convection in porous media plays a central role in numerous large-scale environmental and technological processes. When fluids of differing temperature or concentration occupy a permeable matrix, buoyancy-driven motion provides an efficient mechanism for vertical transport and mixing. In realistic settings, porous materials are rarely homogeneous; instead, they often consist of multiple strata with distinct physical properties. Such \emph{layered porous media} occur naturally in sedimentary basins, hydrogeological formations, and geothermal reservoirs, and are also deliberately engineered in filtration devices and chemical reactors 
\cite{bickle2007modelling, wooding1997convection, huppert2014ar, mckibbin1980jfm, hewitt2014jfm, sahu2017tansp}. 
The interaction of convection with internal interfaces makes the multilayer configuration substantially richer, and more challenging, than the classical single-layer problem.

A common mathematical formulation for these phenomena is the Darcy--Boussinesq system, in which Darcy's law governs the velocity while density variations enter through a buoyancy term, coupled to an advection-diffusion equation for the transported scalar. In a domain composed of several horizontal layers, each region is described by its own permeability coefficient, and the layers are linked through continuity of pressure and normal flux. The permeability of a layer measures its ability to transmit fluid; hence even moderate contrasts between layers can have pronounced influence on convective patterns. In many geophysical applications, one encounters \emph{barrier layers} whose permeability is orders of magnitude smaller than that of the surrounding material. These low-permeability zones act as baffles that impede vertical motion, modify the onset of instability, and may confine vigorous convection to only part of the domain 
\cite{hewitt2022jfm, hewitt2020jfm, huppert2014ar}.

Motivated by this physical background, we examine the limit in which the permeability of one selected layer approaches zero while the remaining layers retain positive permeability. From the modeling perspective, the pressure equation in that layer becomes singular, and the corresponding velocity field formally vanishes. This suggests that the full multilayer convective system should reduce to a coupled problem involving only the permeable layers, together with appropriate transmission conditions. Providing a rigorous justification for such a reduction is far from straightforward: the degeneration of Darcy's law destroys the standard elliptic structure for the pressure and introduces delicate compatibility issues across interfaces. Our main objective is to resolve these difficulties and to place the vanishing-permeability limit on firm mathematical footing.

Specifically, we prove that solutions of the complete multilayer Darcy--Boussinesq system converge to those of the limiting problem, both in terms of $L^{2}$-convergence on arbitrary finite-time intervals and in terms of the convergence of the global attractors that characterize the long-time dynamics. The analysis relies on uniform-in-permeability estimates and on a layer-wise decomposition adapted to horizontal frequency modes, which together compensate for the singular nature of the pressure in the impermeable layer. An explicit solution of a related elliptic model problem further indicates that the rate obtained in the convergence proof is optimal.

The present work complements earlier studies on singular limits in porous convection, such as the vanishing thickness limit of a layer for the Darcy--Boussinesq model \cite{SW2025} and the sharp material interface limit for the same model \cite{DW2025}. In contrast to those results, the limit considered here involves the degeneration of material parameters rather than geometric dimensions, leading to qualitatively different analytical challenges and physical implications.

The remainder of the paper is organized as follows. Section 2 recalls the general multilayer Darcy--Boussinesq model and the associated interfacial boundary conditions. In Section 3 we establish preliminary estimates. Section 4 is devoted to the analysis of the limit sharp interface model with degenerate permeability. In Section 5 we derive uniform in permeability estimates for the model. The convergence results are presented in section 6. Conclusions are provided in section 7. The solution to a closely associated elliptic problem suggesting the optimality of the rate of convergence established in this paper is presented in the appendix.

\section{The mathematic models} 
Recall that convection in a two-dimensional domain $\Omega =  (0,L)\times(-H,0)$ is governed by the following Darcy-Boussinesq system \cite{NB2017},
\begin{equation}\label{Sharp}
\left\{\begin{aligned}
  &\Bu=-\frac{K}{\mu}\left(\nabla P + \rho_{0}(1+\alpha\phi)g\Be_{z}\right),\\
  &\nabla \cdot \Bu =0,\\
  &b\partial_t \phi  + \Bu\cdot\nabla\phi - \nabla \cdot(bD\nabla\phi)=0,
\end{aligned}\right.
\end{equation}
Here $\Bu$, $\phi$, and $P$ are the unknown fluid velocity, concentration, and pressure, respectively; $\rho_{0}$, $\alpha$, $\mu$, $g$ are the constant reference fluid density, constant expansion coefficient, constant dynamic viscosity, and the gravity acceleration constant, respectively; and $\Be_{z}$ stands for the unit vector in the vertical ($z$) direction. In addition, $K$, $b$, $D$ represent the permeability, porosity, and diffusivity respectively. 
The associated convection is often referred to as the Rayleigh-Darcy convection.

When the domain is layered with idealized sharp interfaces, the domain $\Omega$ is divided into $l$ 'layers' or 'strips' by $l-1$ interfaces located at $z= z_i  \in (-H,0), i=1,2,\cdots,l-1$. We denote the $i$-th layer by 
\[\Omega_i : = \{\Bx=(x,z):~z \in (z_{i-1},z_i)\},~~~~\text{ for }i=1,2,\cdots,l.\]
In each layer $\Omega_i$, the permeability, porosity, and diffusivity coefficients are assumed to be constant. Namely,
\[
K=K(\boldsymbol{x})=K_i, \quad b=b(\boldsymbol{x})=b_i, \quad D=D(\boldsymbol{x})=D_i, \quad \boldsymbol{x}\in\Omega_i, \quad 1\leq j\leq l,
\]
for a set of constants $\{K_{i},b_{i},D_{i}\}_{i=1}^{l}$. On the interfaces $z=z_{i}$, we assume
\begin{equation}\label{interface-1}
  \Bu\cdot\Be_{z},\ \phi,\ P \text{ are continuous at } z=z_{i},\ 1\leq  i\leq  l-1,
\end{equation} 

Furthermore, the system \eqref{Sharp} is supplemented with the initial condition
\begin{equation}\label{initialdata}
  \phi|_{t=0}=\phi_{in}
\end{equation}
and the boundary conditions
\begin{equation}\label{bc0}
\phi|_{z=0}=c_0, \quad \phi|_{z=-H}=c_{-H},~~\Bu\cdot \Be_z|_{z=0,-H}= 0,
\end{equation}
together with periodicity in the horizontal directions. 


\section{Preliminary results}
Here we provide a few preliminary results regarding how to homogenize the boundary conditions of the original model, the weak formulation of the problem together with the associated solution spaces, and a few useful estimates associated with the linear principal part of the equations and the solution spaces.

To overcome the nonhomogeneous boundary condition of $\phi$ in \eqref{bc0}, we introduce a smooth function $\phi_b(z;\delta)$ on $[-H,0]$ satisfying
\begin{equation}\label{phi_b1}
\phi_b(z;\delta)= \left\{\begin{aligned}
& c_0, && z = 0,\\
& \frac{c_0+c_{-H}}{2}&&z \in (-H+\delta,-\delta),\\
& c_{-H}, && z=-H
\end{aligned}\right.
\end{equation}
and 
\begin{equation}\label{phi_b2}
|\phi_b'| \leq \frac{c_\Delta}{\delta}, ~~~~|\phi_b''| \leq \frac{2c_\Delta}{\delta^2},
\end{equation}
where $\delta>0$ is a parameter to be determined later and $c_\Delta := |c_0-c_{-H}|$.

Setting
\[ \psi =\phi -  \phi_b,\quad
p = P -\rho_0 g z - \alpha \rho_0 \int_{-H}^z\phi_b(s)\dd s,\]
and assuming $\mu = 1, \rho_{0} \alpha g =1$ and $b=1$ for the remainder of this paper, we obtain the following modified sharp interface model with homogeneous boundary conditions
\begin{equation}\label{Sharps}
      \left\{\begin{aligned}
        &\Bu=- K \left(\nabla p  + \psi \Be_{z}\right),\\
        &\nabla \cdot \Bu =0,\\
        & \partial_t \psi + \Bu\cdot\nabla \psi  +  \phi_b'\Bu \cdot \Be_z   - \nabla\cdot(D\nabla\psi)=D \phi_b'',
      \end{aligned}\right.  
  \end{equation} 
which is supplemented with the interfacial boundary conditions 
\begin{equation}\label{interface1}
  \Bu\cdot\Be_{z},\ \psi,\ p\text{ are continuous at } z=z_{j},\ 1\leqslant j\leqslant l-1,
\end{equation}
together with the initial condition 
\begin{equation}\label{initial}
  \psi|_{t=0}=\psi_{in}:=\phi_{in} -\phi_b.
\end{equation}
The boundary condition is now homogeneous
\begin{equation}\label{bc}
    \psi|_{z=0}=0, \quad \psi|_{z=-H}=0,~~\Bu\cdot \Be_z|_{z=0,-H}= 0,
\end{equation}
and periodicity in the horizontal direction(s).

For $1\leq r \leq \infty$ and $k\in \R$, let $L^r (\Omega)$ and $H^k(\Omega)$ denote the usual Lebesgue space of integrable functions and Sobolev spaces on the domain $\Omega$. The inner product in $L^2(\Omega)$ will be denoted by $(\cdot,\cdot)$. 
Let $H, V$ be the closure of $C^\infty(\overline{\Omega})$ in the $L^2$ and $H^1$ norm, respectively. We denote the dual space of $V$ by $V^*$. The dual product between $V^*$ and $V$ is denoted by $\langle \cdot, \cdot\rangle$. We also denote by $\boldsymbol{H}$ the closure of $\{\Bu \in C^\infty(\overline{\Omega}):~\nabla \cdot \Bu=0\}$ in the $L^2$ norm. Then the weak solution of the sharp interface model \eqref{Sharps}-\eqref{bc} is defined as follows. 
\begin{definition}
Let $\psi_{in}\in  H$ be given, and let $T>0$. A weak solution of the problem \eqref{Sharps}-\eqref{bc} on the interval $[0,T]$ is a triple $(\Bu, \psi, p)$ satisfying the following conditions:
\begin{enumerate}[(1)]
  \item $\psi \in L^\infty(0,T;H) \cap L^2(0,T;V)$, 
  \begin{equation}\nonumber
    (\psi(t_2), \varphi) - ( \psi(t_1), \varphi)  = \int_{t_1}^{t_2}(\Bu\cdot\nabla \psi, \varphi)  +  (\phi_b'\Bu \cdot \Be_z,\varphi)  + (D\nabla\psi,\nabla\varphi)-(D \phi_b',\varphi)\dd t,
  \end{equation}
  for any $t_1,t_2\in [0,T]$ and $\varphi\in V$, and $\psi(0)= \psi_{in}$,
  \item $p \in L^\infty(0,T;H^1(\Omega))$ is a weak solution to the problem
  \begin{equation}\label{pressure}
  \left\{\begin{aligned}
&-\nabla \cdot(K\nabla p) = \nabla\cdot (K\psi \Be_{z}) , ~~~~ \text{ in } \Omega,\\
&\partial_z p(x,-H)  = \partial_z p(x,0)  = 0,
  \end{aligned}\right. 
\end{equation}
  \item $\Bu \in L^2(0,T;\boldsymbol{H}(\Omega))$ satisfies \eqref{Sharps}$_1$.
\end{enumerate}
\end{definition}

The solution to the elliptic problems \eqref{pressure} is understood in the  weak sense:
\begin{equation}\label{weakform-p}
  \int_{\Omega}K(\nabla p+\psi \Be_z) \cdot\nabla q\dd x=0,\quad \text{ for any }q\in C^\infty(\overline{\Omega}).
\end{equation}
Once the piecewise constant permeability is positive in each layer, the existence of the elliptic problem \eqref{pressure} for any $\psi\in H$ is guaranteed by the Lax-Milgram theorem. Uniqueness can be obtained if we restrict to the mean zero subspace of $H^1$. For piecewise smooth $\psi$, one can show that $p$ is also piecewise smooth. More specifically, if $\psi\in V$ , then $p$ is piecewise $H^2$. 
 
For higher regularity of the solutions to the sharp interface model \eqref{Sharps}, we recall the following space, which is introduced in \cite{CNW2025}.
\begin{definition}
Define
\begin{equation}\nonumber
W = \{\psi \in V:~\partial_x \psi \in H^1(\Omega),~D\partial_z \psi \in H^1(\Omega)\}
\end{equation}
endowed with norm
\begin{equation}\nonumber
  \|\psi\|_W = \|\psi\|_{H^1(\Omega)}+\|\partial_x \psi\|_{H^1(\Omega)}+\|D\partial_z \psi\|_{H^1(\Omega)}.
\end{equation}
\end{definition} 

The weighted space $W$ is different from the classical $H^2$ space in general unless unless $D$ is smooth in $\Omega$.
Similar to the classical Sobolev space $H^2$, one could also establish the Gagliardo-Nirenberg type inequality for the weighted space $W$. 
\begin{lemma}\label{lemmaA1}
  For any $\psi \in V$, it holds that
   \begin{equation}\label{A1-0}
  \|\psi\|_{L^4(\Omega)} \leq C\|\nabla\psi\|_{L^2(\Omega)}^\frac12\|\psi\|_{L^2(\Omega)}^\frac12 ,
  \end{equation}
Furthermore, we have 
  \begin{equation}\label{A1-1}
  \|\nabla \psi\|_{L^4(\Omega)} \leq C\|\nabla\psi\|_{L^2(\Omega)}^\frac12\|\psi\|_{W}^\frac12, ~~~~\text{ for any }\psi\in W.
  \end{equation}  
Here the constant  $C$ depends only on $\Omega$ and $D_i$.
\end{lemma}
We also recall the equivalence between $\|\cdot\|_W$ and the norm associated with the operator $\mathcal{L}\psi:= -\nabla\cdot (D\nabla\psi)$. One may refer to \cite{CNW2025} for the detailed proof.
\begin{lemma}\label{lemma-equivalence}
 There exist two constants $C_l, C_u>0$ depending only on $H$ and $D_i$ such that
\begin{equation} \label{A4-0-1}
C_l\|\mathcal{L}\psi\|_{L^2(\Omega)} \leq \|\psi\|_{W} \leq C_u\|\mathcal{L}\psi\|_{L^2(\Omega)},~~~\text{ for any }\psi \in W.
\end{equation}  
\end{lemma} 

Following \cite{CNW2025}, one can obtain the existence, uniqueness, and regularity of solutions to the sharp interface model \eqref{Sharps}-\eqref{bc} as follows.
\begin{proposition}\label{prop1}
Assume that the permeability $K$ is positive in each layer.  For each $\psi_{in}\in H$, there exists a unique global weak solution $(\Bu,\psi, p)$ to problem \eqref{Sharps}-\eqref{bc}. In particular, the map $\psi_{in}\mapsto \psi(t)$ is continuous in $\psi_{in}$ and $t$ in $H$. If, furthermore, $\psi_{in}\in V$,  we have
\begin{equation}
\psi \in L^\infty([0,T);V) \cap L^2(0,T;W),~~~~~ \text{for any }T>0.
\end{equation}
\end{proposition}

\section{Sharp interface model with degenerate permeability}
The purpose of this section is to investigate the limit model with degenerate permeability. The associated elliptic problem for the pressure is degenerate in this case. The difficulty associated with degeneracy is resolved by decomposing the solutions via the different layers and in terms Fourier modes in the horizontal direction(s). 

If the permeability of a certain layer $\Omega_j$ tends to zero, it is natural to believe that the system \eqref{Sharps}-\eqref{bc} converges to the  corresponding system with zero permeability. However, at this point, the equation \eqref{pressure} becomes a degenerate elliptic system, whose solvability is no longer guaranteed by the straightforward application of Lax-Milgram theorem. This section is dedicated to establishing the well-posedness theorem of the problem \eqref{Sharps}-\eqref{bc} in the case of degenerate permeability.

For convenience, we relabel the layer with zero permeability as
\[ \Omega_0: = \Omega_j = \{\Bx \in \Omega:~ z_{j-1} <z<z_{j}\} \]
and denote 
\[  \Omega_- := \{\Bx \in \Omega:~-H <z<z_{j-1}\},\quad \Omega_+ := \{\Bx \in \Omega:~z_j <z <0\}.\]  

\begin{lemma}\label{degenerate-p}
Suppose that the piecewise constant permeability satisfies $K=0$ in $\Omega_0$ and $K> 0$ in $\Omega_\pm$. Then for any $\psi\in V$, the degenerate elliptic problem \eqref{pressure} admits at least one solution $p\in H^1(\Omega)$ with zero mean value. Furthermore, the solution $p$ satisfies
\begin{equation}\nonumber
  \|p\|_{H^1(\Omega)}\leq C\|\psi\|_{L^2(\Omega)}.
\end{equation}
\end{lemma}
\begin{proof}
The $x$-periodic function $\psi$ on $\Omega$ can be written as the Fourier expansion 
\begin{equation}\nonumber
\psi(x,z) = \psi_0(z) +  \sum_{n\in \mathbb{Z},n\neq0 }\psi_n(z) e^{\ii n \frac{2\pi}{L} x} =: \psi_0(z)+ \psi _{\neq}(x,z),
\end{equation}
where 
\begin{equation}\nonumber
\psi_n (z) = \frac{1}{L}\int_{0}^{L} \psi(x,z) e^{\ii n \frac{2\pi}{L} x} \dd x. 
\end{equation}

For the zero mode, it is easy to see that 
\begin{equation} \label{degenerate-p-0}
 p_0'(z) = -\psi_0(z),\quad\text{ together with }\quad \int_{-H}^0 p_0\dd z =0,
\end{equation}
determines a  solution to the problem
\begin{equation}\label{degenerate-p-1}
\left\{\begin{aligned}
&-\nabla \cdot (K\nabla p_0) = \nabla \cdot (K\psi_0\Be_z),\quad&& \text{ in }\Omega,\\
&\partial_zp_0=  0,&& \text{ on }\partial\Omega.
\end{aligned}\right.
\end{equation}
By virtue of Poincar\'e inequality, we have 
\begin{equation}\label{degenerate-p-2}
\|p_0\|_{H^1(\Omega)} \leq \|\nabla p_0\|_{L^2(\Omega)} \leq C\|\psi_0\|_{L^2(\Omega)}.
\end{equation} 

Now we construct the nonzero mode of $p$ in the subdomains $\Omega_\pm$ and $\Omega_0$, respectively. 
Since $K$ is strictly positive in $\Omega_\pm$, it follows from Lax-Milgram theorem that the problem
  \begin{equation}\nonumber
\left\{\begin{aligned}
&-\nabla \cdot(K\nabla p_{\neq,\pm}) = \nabla\cdot (K\psi_{\neq} \Be_{z}) , ~~~~ &&\text{ in } \Omega_\pm,\\
& \partial_z p_{\neq}= -\psi_{\neq} ,  &&\text{ on } \partial\Omega_\pm 
\end{aligned}\right.
\end{equation}
admits a unique solution $p_{\neq,\pm}\in H^1(\Omega_\pm)$, which satisfies
\begin{equation}\nonumber
\int_{0}^L p_{\neq,\pm} (x,z)\dd x =0\quad \text{ for any }z,
\end{equation}
and 
\begin{equation}\nonumber
\|\nabla p_{\neq,\pm}\|_{L^2(\Omega_\pm)} \leq C\|\psi_{\neq}\|_{L^2(\Omega_\pm)}.
\end{equation}
By Poincar\'e inequality, we have 
\begin{equation}\label{degenerate-p-4}
\| p_{\neq,\pm}\|_{H^1(\Omega_\pm)} \leq C\|\psi_{\neq}\|_{L^2(\Omega_\pm)}.
\end{equation} 

On the other hand, according to \cite[Theorem 8.3]{GT}, the Dirichlet problem of Poisson equation
\begin{equation}\label{degenerate-p-6}
\left\{\begin{aligned}
&-\Delta p _{\neq,0} = \partial_z\psi _{\neq}     ,\quad&& \text{ in }\Omega_0,\\
&p _{\neq,0}= p_{\neq,\pm}&&\text{ on }\partial\Omega_0 \cap \partial\Omega_\pm.
\end{aligned}\right. 
\end{equation}
 admits a unique solution $p _{\neq,0}\in H^1(\Omega_0)$ satisfying 
\begin{equation} \label{degenerate-p-7}
\begin{aligned}
\|p_{\neq,0}\|_{H^1(\Omega_0)} \leq&~ C(\|\psi_{\neq}\|_{L^2(\Omega_0)} + \|p_{\neq,+}\|_{H^\frac12(\partial\Omega_+\cap \partial\Omega_0)} + \|p_{\neq,-}\|_{H^\frac12 (\partial\Omega_-\cap\partial\Omega_0)}) \\
\leq&~ C(\|\psi_{\neq}\|_{L^2(\Omega_0)} + \|p_{\neq,+}\|_{H^1(\Omega_+)} + \|p_{\neq,-}\|_{H^1 (\Omega_-)})\\
\leq&~ C \|\psi_{\neq}\|_{L^2(\Omega)}. 
\end{aligned}
\end{equation}
Then, by virtue of \eqref{degenerate-p-6}, the function
\begin{equation}\nonumber
p_{\neq}: =\left\{\begin{aligned}
  &p_{\neq,+},\quad &&\text{ in }\Omega_+,\\
  &p_{\neq,0},&&\text{ in }\Omega_0,\\
  &p_{\neq,-},&&\text{ in }\Omega_-,
\end{aligned}\right.
\end{equation}
is continuous at the interface and satisfies the problem 
\begin{equation}\nonumber
\left\{\begin{aligned}
&-\nabla \cdot(K\nabla p_{\neq,\pm}) = \nabla\cdot (K\psi_{\neq} \Be_{z}) , ~~~~ &&\text{ in } \Omega ,\\
& \partial_z p_{\neq}=0 ,  &&\text{ on } \partial\Omega.
\end{aligned}\right.
\end{equation}
Thus $p := p_0+ p_{\neq}$ is a solution to the problem \eqref{pressure} with 
\begin{equation}\nonumber
\int_\Omega p\dd x\dd z  = L \int_{-H}^0 p_0\dd z=0.
\end{equation}
 Finally, combining \eqref{degenerate-p-2}-\eqref{degenerate-p-4} and \eqref{degenerate-p-7} gives
\begin{equation}\nonumber
\|p\|_{H^1(\Omega)} \leq C \|\psi \|_{L^2(\Omega)} .
\end{equation}
This finishes the proof of this lemma.
\end{proof}
\begin{remark}
The solutions to the degenerate problem \eqref{pressure} are not unique because the restriction $p|_{\Omega_0}$ can be arbitrary, provided that $p$ is continuous across the interface $\partial\Omega_0$. However, $\nabla p|_{\Omega_\pm}$ is uniquely determined by $\psi$. 
\end{remark}

Then the existence, uniqueness and regularity of the solution to the degenerate sharp interface model can be established via an approach analogous to that in \cite{CNW2025}. 

\begin{proposition}\label{existence-degenerate}
Suppose that the permeability $K$ vanishes in the layer $\Omega_0$ and $K>0$ in $\Omega_\pm$.  For each $\psi_{in}\in H$, there exists a global solution $(\Bu,\psi, p)$ to problem \eqref{Sharps}-\eqref{bc}. In particular, the restriction $(\Bu,\psi,\nabla p|_{\Omega_\pm})$ is unique, and the map $\psi_{in}\mapsto \psi(t)$ is continuous in $\psi_{in}$ and $t$ in $H$. Furthermore,  if $\psi_{in} \in V$, then the solution satisfies 
\begin{equation}\nonumber
\psi \in L^\infty(0,T;V) \cap L^2(0,T;W),~~~~~ \text{for any }T>0.
\end{equation} 
\end{proposition}

 \section{Uniform estimate for low permeability}  
This section is dedicated to establishing the uniform (in permeability) estimates for solutions to the problem (\ref{Sharps}-\ref{bc}), which are crucial for the subsequent analysis of long-time behavior and convergence. We begin by proving in Proposition \ref{L^2} the uniform $L^2$ estimate of $\psi^\varepsilon$.
This uniform in time estimate relies on a careful choice of $\delta$ in choosing our background profile $\phi_b$. Building upon this and leveraging the uniform Gronwall inequality, Proposition~\ref{H^1} provides higher regularity estimates, including uniform-in-time $H^1$ bounds for large times and finite-time estimates depending on the initial data in $V$. Finally, Proposition~\ref{uniform-p} establishes uniform $H^1$ estimates for the pressure $p^\varepsilon$.  Similar to the analysis of the limit model, the pressure estimates are derived via decomposition of the solution in the zero-horizontal model as well as in the horizontal layers.

For any $\varepsilon\ge 0$, we denote the overall permeability in $\Omega$ by $K^\varepsilon$, where the superscript indicates that the permeability $K^\varepsilon$ equals to $\varepsilon$ in the layer $\Omega_0$, while $K^\varepsilon$ is positive and independent of $\varepsilon$ in $\Omega_\pm$.
\begin{proposition}\label{L^2}
Let $\delta$ be chosen so that \eqref{delta_2} is satisfied. Suppose that $\psi_{in}^\varepsilon\in L^2(\Omega)$ and $(\Bu^\varepsilon,\psi^\varepsilon,p^\varepsilon)$ is the solution to the problem \eqref{Sharps}-\eqref{bc} with permeability $K^\varepsilon$. Then for any $t>0$, it holds that 
\begin{equation}\label{L^2-1}
\|\psi^\varepsilon(t)\|_{L^2(\Omega)}^2 \leq \|\psi_{in}^\varepsilon\|_{L^2(\Omega)}^2 e^{-\frac{\min_i D_i}{H^2}t} +\frac{M_1H^2}{\min_i D_i}\left(1- e^{-\frac{\min_i D_i}{H^2}t}\right)
\end{equation}
and 
\begin{equation}\label{L^2-1-1}
\int_0^t \|\sqrt{D} \nabla \psi^\varepsilon(s)\|_{L^2(\Omega)}^2\dd s  \leq M_1 t +\| \psi_{in}^\varepsilon\|_{L^2(\Omega)}^2,
\end{equation}
where 
\begin{equation}\label{M1-delta}
M_1 = \frac{8 c_\Delta^2 L^2}{\delta}\frac{ (\max_i D_i )^2}{\min_i  D_i }.
\end{equation}

In particular, for any  $t\ge T_1:=T_1(\|\psi_{in}^\varepsilon\|_{L^2(\Omega)}) =\frac{2H^2}{\min_i D_i}\ln \|\psi_{in}^\varepsilon\|_{L^2(\Omega)}$, one has 
\begin{equation}\label{L^2-1-2}
\|\psi^\varepsilon(t) \|_{L^2(\Omega)}^2\leq  \frac{M_1H^2}{\min_i D_i}+1,~~~~~~\int_t^{t+1} \|\sqrt{D}\nabla\psi^\varepsilon \|_{L^2(\Omega)}^2 \dd s \leq  \frac{M_1H^2 }{\min_i D_i } +1.
\end{equation}
\end{proposition}
\begin{proof}
Testing the equation \eqref{Sharps}$_3$ by $\psi^\varepsilon$ and using integration by parts, one has
\begin{equation}\label{L^2-2}
\frac12  \frac{\dd}{\dd t} \|\psi^\varepsilon\|_{L^2(\Omega)}^2 + \|\sqrt{D}\nabla\psi^\varepsilon\|_{L^2(\Omega)}^2 =  -   ( \phi_b' u_z^\varepsilon, \psi^\varepsilon ) + (D \phi_b'', \psi^\varepsilon ) .
\end{equation} 
By virtue of \eqref{phi_b1} and \eqref{phi_b2},  we have 
\begin{equation}\label{L^2-3}
\left|( \phi_b' u_z^\varepsilon, \psi^\varepsilon ) \right| = \left|\int_{\Omega_\delta}  \phi_b' u_z^\varepsilon \psi^\varepsilon \dd x\dd z\right| \leq  c_\Delta \delta^{-1}\|u_z^\varepsilon\|_{L^2(\Omega_\delta)} \|\psi^\varepsilon\|_{L^2(\Omega_\delta)},
\end{equation} 
where $\Omega_\delta := \Omega \cap (\{-H<z<-H+\delta\} \cup \{-\delta <z<0\})$. Noting that both $u_z^\varepsilon$ and $\psi^\varepsilon$ vanish on  the boundary $\{z=0\}\cup \{z=-H\}$, one uses Poincar\'e inequality to obtain
\begin{equation}\label{L^2-3-1}
\|u_z^\varepsilon\|_{L^2(\Omega_\delta)} \leq   \delta\|\partial_z u_z^\varepsilon\|_{L^2(\Omega_\delta)},~~\|\psi^\varepsilon\|_{L^2(\Omega_\delta)} \leq   \delta\|\partial_z \psi^\varepsilon\|_{L^2(\Omega_\delta)},
\end{equation}
and thus,
\begin{equation}\label{L^2-4}
\left|( \phi_b' u_z^\varepsilon, \psi^\varepsilon ) \right|\leq c_\Delta\delta\|\partial_z^\varepsilon u_z^\varepsilon\|_{L^2(\Omega_\delta)}\|\partial_z \psi^\varepsilon\|_{L^2(\Omega_\delta)} \leq  c_\Delta\delta \|\partial_x u_x^\varepsilon\|_{L^2(\Omega)} \|\partial_z \psi^\varepsilon\|_{L^2(\Omega)}.
\end{equation} 
Here we have used the fact that $\partial_z u_z^\varepsilon = - \partial_x u_x^\varepsilon$ due to the incompressibility condition $\nabla \cdot \Bu^\varepsilon =0$. Since the permeability coefficient $K^\varepsilon$ is independent of $x$, we have from \eqref{Sharps}$_1$ that
\begin{equation}\label{L^2-5}
\|\partial_x \Bu^\varepsilon\|_{L^{2}(\Omega)} \leq   \|K^\varepsilon\partial_x \nabla p^\varepsilon\|_{L^{2}(\Omega)}+ \|K^\varepsilon\partial_x  \psi^\varepsilon\|_{L^2(\Omega)},
\end{equation}
while the pressure $p^\varepsilon$ is determined by the elliptic problem  
 \begin{equation}\nonumber\label{pressure_x}
  \left\{\begin{aligned}
&-\nabla\cdot(K^\varepsilon\nabla  \partial_x p^\varepsilon) = \nabla\cdot (K^\varepsilon \partial_x \psi^\varepsilon \Be_z) ,  ~~~~\text{ in } \Omega,\\
&\partial_z \partial_x p^\varepsilon(x,-H)  = \partial_z \partial_x p^\varepsilon(x,0)  = 0.
  \end{aligned}\right. 
\end{equation}
Then it follows from the standard elliptic estimate  that we have
\begin{equation}\label{L^2-5-1}
\|\sqrt{K^\varepsilon}\partial_x\nabla p^\varepsilon\|_{L^2(\Omega)} \leq    \|\sqrt{K^\varepsilon}\partial_x\psi^\varepsilon\|_{L^2(\Omega)}.
\end{equation}
This, together with \eqref{L^2-5}, yields
\begin{equation}\label{L^2-5-2}
\begin{aligned}
\|\partial_x \Bu^\varepsilon\|_{L^{2}(\Omega)} \leq &~\|\sqrt{K^\varepsilon}\|_{L^\infty(\Omega)}(\|\sqrt{K^\varepsilon}\partial_x\nabla p^\varepsilon\|_{L^2(\Omega)}+  \|\sqrt{K^\varepsilon}\partial_x\psi^\varepsilon\|_{L^2(\Omega)})\\
\leq&~  2 \max_i  K_i  \|\partial_x\psi^\varepsilon\|_{L^2(\Omega)}.
\end{aligned}
\end{equation} 
Hence,
\begin{equation}\label{L^2-6}
\left|( \phi_b' u_z^\varepsilon, \psi^\varepsilon ) \right|\leq 2c_\Delta \delta \max_i  K_i    \|\nabla \psi^\varepsilon\|_{L^2(\Omega)}^2.
\end{equation}
On the other hand, using H\"older's inequality, Young's inequality and \eqref{L^2-3-1}, we have
\begin{equation}\label{L^2-7}
\begin{aligned}
\left|(D \phi_b'',\psi ^\varepsilon)\right| \leq  &~ 2c_\Delta \delta^{-2} \max_i D_i  |\Omega_\delta|^\frac12 \| \psi^\varepsilon\|_{L^2(\Omega_\delta)}  \\
\leq  &~ 2c_\Delta L^\frac12  \delta^{-\frac12} \max_i D_i    \| \partial_z\psi^\varepsilon\|_{L^2(\Omega_\delta)} \\ 
\leq &~ \frac{1}{4}\min_i  D_i \|\nabla \psi^\varepsilon\|_{L^2(\Omega)}^2+  \frac{4 c_\Delta^2 L }{\delta}\frac{ (\max_i D_i )^2}{\min_i  D_i }.
\end{aligned}
\end{equation}

Substituting \eqref{L^2-6}-\eqref{L^2-7} into \eqref{L^2-2} and choosing 
\begin{equation}\label{delta_2}
\delta \leq \frac{ \min_i D_i}{8  c_\Delta \max_i  K_i },
\end{equation} 
one has 
\begin{equation}\label{L^2-8}
 \frac{\dd}{\dd t} \|\psi^\varepsilon\|_{L^2(\Omega)}^2  +   \|\sqrt{D}\nabla\psi^\varepsilon\|_{L^2(\Omega)}^2 \leq  \frac{8 c_\Delta^2 L }{\delta}\frac{ (\max_i D_i )^2}{\min_i  D_i }=:M_1.
\end{equation}  
Integrating \eqref{L^2-8} over $[0,T]$ gives \eqref{L^2-1-1}. Furthermore, due to Poincar\'e inequality, we also infer from \eqref{L^2-8} that 
\begin{equation}\nonumber
 \frac{\dd}{\dd t} \|\psi^\varepsilon\|_{L^2(\Omega)}^2  + \frac{\min_i D_i}{H^2} \|\psi^\varepsilon\|_{L^2(\Omega)}^2 \leq M_1.
\end{equation}  
Using the classical Gr\"onwall inequality, we get \eqref{L^2-1}. In particular, it follows from \eqref{L^2-1} that
\begin{equation}\label{L^2-10}
\|\psi^\varepsilon(t)\|_{L^2(\Omega)}^2 \leq \frac{M_1H^2}{\min_i D_i}+1,~~~~\text{ for any }t \ge T_1=\frac{2H^2}{\min_i D_i}\ln  \|\psi_{in}^\varepsilon\|_{L^2(\Omega)}.
\end{equation}
Finally, for $t\ge T_1$, we integrate the energy inequality \eqref{L^2-8} over $[t,t+1]$ to obtain
\begin{equation}\label{L^2-11}
\|\psi^\varepsilon(t+1) \|_{L^2(\Omega)}^2+  \int_t^{t+1} \|\sqrt{D}\nabla\psi^\varepsilon \|_{L^2(\Omega)}^2 \dd s \leq \|\psi^\varepsilon(t) \|_{L^2(\Omega)}^2 \leq \frac{M_1H^2}{\min_i D_i}+1.
\end{equation}
Combining \eqref{L^2-10} and \eqref{L^2-11}, one proves \eqref{L^2-1-2} and finishes the proof of the proposition.
\end{proof}

\begin{proposition}\label{H^1}
Suppose that $\psi_{in}^\varepsilon\in H$ and $(\Bu^\varepsilon,\psi^\varepsilon,p^\varepsilon)$ is the solution to the problem \eqref{Sharps}-\eqref{bc} with permeability $K^\varepsilon$. Then for any $t\ge T_1+1$, it holds that
\begin{equation} \label{H^1-0}
\|\sqrt{D}\nabla \psi^\varepsilon(t)\|_{L^2(\Omega)}^2   \leq  M_5 
\end{equation}
where the constant $M_5$ given in \eqref{M5} is independent of $\varepsilon$ and the initial data. If, in addition, $\psi_{in}^\varepsilon\in V$, we have  
\begin{equation}\label{H^1-1} 
  \|\nabla \psi^\varepsilon(t)\|_{L^2(\Omega)} ^2 \leq M_6(\|\psi_{in}^\varepsilon\|_{H^1(\Omega)}), 
\end{equation}
for any $t\ge 0$.
\end{proposition}
\begin{proof}
Multiplying \eqref{Sharps}$_3$ by $\mathcal{L}\psi^\varepsilon = -\nabla\cdot (D\nabla\psi^\varepsilon)$ and integrating result equation over $\Omega$, we obtain
\begin{equation}\label{H^1-3}
\begin{aligned}
&~\frac12  \frac{\dd}{\dd t} \|\sqrt{ D}\nabla \psi^\varepsilon\|_{L^2(\Omega)}^2 + \|\mathcal{L}\psi^\varepsilon\|_{L^2(\Omega)}^2 \\
= &~ (\Bu^\varepsilon\cdot\nabla  \psi^\varepsilon,\mathcal{L}\psi^\varepsilon) + (\phi_b' u_z, \mathcal{L}\psi^\varepsilon) -  (D\phi_b'',\mathcal{L}\psi^\varepsilon)   .
\end{aligned}
\end{equation}
Using H\"older's inequality, Young's inequality and \eqref{phi_b2}, one has 
\begin{equation}\label{H^1-4}
\begin{aligned}
|(D \phi_b'',\mathcal{L}\psi^\varepsilon)|\leq  &~ \frac{2c_\Delta}{\delta^2}\max_i D_i  \int_{\Omega_{\delta}} | \mathcal{L}\psi^\varepsilon|  \dd \Bx  \\
\leq  &~ \frac{2c_\Delta}{\delta^2}\max_i D_i   |\Omega_\delta|^\frac12 \|\mathcal{L}\psi^\varepsilon\|_{L^2(\Omega_\delta)}  \\
\leq  &\frac{1}{4}\|\mathcal{L}\psi^\varepsilon\|_{L^2(\Omega)}^2 +4 c_\Delta^2 L \delta^{-3} (\max_i D_i )^2.
\end{aligned}
\end{equation}
In a way similar to \eqref{L^2-3}-\eqref{L^2-6}, one may use Young's inequality to deduce
\begin{equation}\label{H^1-5}
\begin{aligned}
\left|(\phi_b' u_z^\varepsilon, \mathcal{L}\psi^\varepsilon)\right| 
\leq  &~ c_\Delta \delta^{-1} \|u_z^\varepsilon\|_{L^2(\Omega_\delta)}\|\mathcal{L}\psi^\varepsilon\|_{L^2(\Omega_\delta)}\leq  c_\Delta \|\partial_x u_x^\varepsilon\|_{L^2(\Omega)}\|\mathcal{L}\psi^\varepsilon\|_{L^2(\Omega)}\\
\leq  &~  2  c_\Delta  \max_i  K_i    \|\partial_x \psi^\varepsilon\|_{L^2(\Omega)}\|\mathcal{L}\psi^\varepsilon\|_{L^2(\Omega)}\\
\leq  &~ \frac{1}{4}\|\mathcal{L}\psi^\varepsilon\|_{L^2(\Omega)}^2+ 4c_\Delta^2 (\max_i  K_i)^2  \|\nabla\psi^\varepsilon\|_{L^2(\Omega)}^2.
\end{aligned}
\end{equation} 
Furthermore, using H\"older's inequality, Young's inequality and Lemmas \ref{lemmaA1}-\eqref{lemma-equivalence}, we have 
\begin{equation}\label{H^1-6}
\begin{aligned}
\left|(\Bu^\varepsilon\cdot\nabla \psi^\varepsilon, \mathcal{L}\psi^\varepsilon)\right| 
\leq  &~ \|\Bu^\varepsilon\|_{L^4(\Omega)} \|\nabla \psi^\varepsilon\|_{L^4(\Omega)}  \|\mathcal{L} \psi^\varepsilon\|_{L^2(\Omega)}\\
\leq  &~ C\|\Bu^\varepsilon\|_{L^4(\Omega)} \|\nabla \psi^\varepsilon\|_{L^2(\Omega)}^\frac12 \|\mathcal{L} \psi^\varepsilon\|_{L^2(\Omega)}^\frac32\\ 
\leq  &~\frac{1}{4}\|\mathcal{L} \psi^\varepsilon\|_{L^2(\Omega)}^2+  C\|\nabla \psi^\varepsilon\|_{L^2(\Omega)}^2\sum_{i}\|\Bu^\varepsilon\|_{L^4(\Omega_i)}^4 .
\end{aligned}
\end{equation}
Similar to \eqref{L^2-5-2}, one has also 
\begin{equation}\label{H^1-7-1} 
   \| \Bu^\varepsilon\|_{L^2(\Omega )}
\leq  2\max_{i} K_i \| \psi^\varepsilon\|_{L^2(\Omega)}. 
\end{equation} 
On the other hand, if $\psi^\varepsilon \in V$, the press $p^\varepsilon$ is piecewise $H^2$ and satisfies
\[ -\Delta p^\varepsilon = \partial_z\psi^\varepsilon,\quad \text{ in each }\Omega_i.\]
Using \eqref{L^2-5-1}, one has 
\begin{equation}\nonumber
  \|\sqrt{K^\varepsilon}\partial_{z}^2 p^\varepsilon\|_{L^2(\Omega_i) }\leq   \|\sqrt{K^\varepsilon}\partial_{x}^2 p^\varepsilon\|_{L^2(\Omega_i) }+ \|\sqrt{K^\varepsilon}\partial_{z}\psi^\varepsilon\|_{L^2(\Omega_i) } 
  \leq 2\|\sqrt{K^\varepsilon}\nabla\psi^\varepsilon\|_{L^2(\Omega) },
\end{equation}
for any $i$. This, together with \eqref{L^2-5-1}, gives
\begin{equation} \label{H^1-7-2}
\begin{aligned}
  \|\nabla\Bu^\varepsilon\|_{L^2(\Omega_i)} \leq&~
   \|K^\varepsilon\nabla^2 p^\varepsilon\|_{L^2(\Omega_i)} +\| K^\varepsilon \nabla \psi^\varepsilon\|_{L^2(\Omega_i)}\\ 
  \leq&~\sqrt{K_i}(\|\sqrt{K^\varepsilon} \nabla^2 p^\varepsilon\|_{L^2(\Omega_i)} +\|\sqrt{K^\varepsilon}\nabla \psi^\varepsilon\|_{L^2(\Omega_i)})\\
  \leq&~C \sqrt{K_i}\|\sqrt{K^\varepsilon}\nabla \psi^\varepsilon\|_{L^2(\Omega)}\\
  \leq&~C\max_i K_i \| \nabla \psi^\varepsilon\|_{L^2(\Omega)}.
\end{aligned}
\end{equation} 
By virtue of Gagliardo-Nirenberg inequality, Poincar\'e inequality and \eqref{H^1-7-1}-\eqref{H^1-7-2}, one has
\begin{equation}\label{H^1-7}
\begin{aligned}
\|\Bu^\varepsilon\|_{L^4(\Omega_i)}\leq&~  C(\|\Bu^\varepsilon\|_{L^2(\Omega_i)}^\frac12\|\nabla\Bu^\varepsilon\|_{L^2(\Omega_i)}^\frac12 + \|\Bu^\varepsilon\|_{L^2(\Omega_i)}) \\
\leq&~C \max_i K_i \|\nabla\psi^\varepsilon\|_{L^2(\Omega)}^\frac12\|\psi^\varepsilon\|_{L^2(\Omega)}^\frac12.
\end{aligned}
\end{equation}
Substituting the estimates \eqref{H^1-4}-\eqref{H^1-6} into \eqref{H^1-3} and using \eqref{H^1-7}, one obtains
\begin{equation}\label{H^1-8} 
\begin{aligned}
&~ \frac{\dd}{\dd t} \|\sqrt{D}\nabla \psi^\varepsilon\|_{L^2(\Omega)}^2  + \frac{1}{2} \|\mathcal{L}\psi^\varepsilon\|_{L^2(\Omega)}^2 \\
\leq& ~ M_2 + (M_3  + M_4\|\psi^\varepsilon\|_{L^2(\Omega)}^2 \|\sqrt{D}\nabla \psi^\varepsilon\|_{L^2(\Omega)}^2)\|\sqrt{D} \nabla \psi^\varepsilon\|_{L^2(\Omega)}^2  , 
\end{aligned}
\end{equation} 
with constants
 \begin{equation}\label{M2M3}
  M_2= 8 c_\Delta^2 L \delta^{-3} (\max_i D_i )^2,\quad ~~M_3 = \frac{8 c_\Delta^2(\max_i  K_i)^2   }{ \min_i  D_i},
\end{equation}
and 
\begin{equation}\label{M4}
M_4 = C (\max_i K_i)^4 .
\end{equation}   

Set
\begin{equation}\nonumber
y(t) := \|\sqrt{D}\nabla \psi^\varepsilon\|_{L^2(\Omega)}^2, ~~~g(t) := M_3+M_4\|\psi^\varepsilon\|_{L^2(\Omega)}^2 \|\sqrt{D}\nabla \psi^\varepsilon\|_{L^2(\Omega)}^2,~~~h(t) := M_2.
\end{equation}
Then, for any $t>0$, we have
\begin{equation}\nonumber
  y'(t) \leq g(t) y(t) + h(t) \quad\text{ and} \quad
\int_{t}^{t+1} h(s)\dd s = M_2.
\end{equation}
Furthermore, if $t\ge T_1$, it follows from Proposition \ref{L^2} that
\begin{equation}\nonumber
  \int_{t}^{t+1} y(s)\dd s = \int_{t}^{t+1} \|\sqrt{D}\nabla \psi^\varepsilon\|_{L^2(\Omega)}^2 \dd s \leq \frac{M_1H^2 }{\min_i D_i} +1
\end{equation} 
and 
\begin{equation}\nonumber
\begin{aligned}
      \int_{t}^{t+1} g(s)\dd s \leq&~ M_3 +M_4\sup_{s\in[t,t+1]} \|\psi^\varepsilon(s)\|_{L^2(\Omega)}^2\int_{t}^{t+1}   \|\sqrt{D}\nabla \psi^\varepsilon\|_{L^2(\Omega)}^2 \dd s \\
      \leq&~ M_3+ M_4 \left(\frac{M_1H^2}{\min_i D_i}+1\right)^2.
\end{aligned}
\end{equation} 
Then an application of the uniform Gr\"onwall lemma \cite[Lemma 1.1]{Temam1997} gives
\begin{equation}\label{M5}
y(t+1)  \leq  \left(M_2+  \frac{M_1H^2}{\min_i D_i}+1\right) \operatorname{exp}\left[M_3+M_4\left(\frac{M_1H^2}{\min_i D_i}+1\right)^2\right]=:M_5,
\end{equation}
for any $t\ge T_1$.  

Finally, we assume that $\psi_{in}^\varepsilon \in V$. With the aid of \eqref{L^2-1} and \eqref{L^2-1-1}, one could apply Gr\"onwall's inequality to \eqref{H^1-8} and deduce
\begin{equation}\nonumber
\begin{aligned}
 &~\|\sqrt{D}\nabla \psi^\varepsilon(t)\|_{L^2(\Omega)}^2 \\
 \leq&~ \|\sqrt{D}\nabla \psi_{in}^\varepsilon\|_{L^2(\Omega)}^2\exp\left(\int_0^t (M_3+M_4\|\psi^\varepsilon\|_{L^2(\Omega)}^2 \|\sqrt{D}\nabla \psi^\varepsilon\|_{L^2(\Omega)}^2) \dd s\right) \\
 &~+ \int_0^t M_2  \exp\left(\int_\tau^t (M_3+M_4\|\psi\|_{L^2(\Omega)}^2 \|\sqrt{D}\nabla \psi^\varepsilon\|_{L^2(\Omega)}^2) \dd s\right) \dd \tau\\
 \leq&~ (\|\sqrt{D}\nabla \psi_{in}^\varepsilon\|_{L^2(\Omega)}^2+  M_2 t)\exp\left(M_3 t +M_4\int_0^t \|\psi^\varepsilon\|_{L^2(\Omega)}^2 \|\sqrt{D}\nabla \psi^\varepsilon\|_{L^2(\Omega)}^2  \dd s\right)\\
 \leq&~ (\|\sqrt{D}\nabla \psi_{in}^\varepsilon\|_{L^2(\Omega)}^2+ M_2 t)\\
 &\times\exp\left[M_3 t +M_4\left(\frac{M_1H^2}{\min_i D_i}+\|\psi_{in}^\varepsilon\|_{L^2(\Omega)}^2\right)\left(M_1t + \|\psi_{in}^\varepsilon\|_{L^2(\Omega)}^2\right)\right]\\
 =&:M_5'(t,\|\psi_{in}^\varepsilon\|_{H^1(\Omega)}).
 \end{aligned}
\end{equation} 
This, together with the uniform estimate \eqref{H^1-1} in long time, gives  
\begin{equation}\nonumber
\begin{aligned}
\| \nabla \psi^\varepsilon(t)\|_{L^2(\Omega)}^2\leq&~ (\min_iD_i)^{-\frac12}\| \sqrt{D}\nabla \psi^\varepsilon(t)\|_{L^2(\Omega)}^2\\
\leq&~ (\min_iD_i)^{-\frac12}\max\{M_5,M_5'(T_1+1,  \|\psi_{in}^\varepsilon\|_{H^1(\Omega)})\}  \\
=&:M_6(\|\psi_{in}^\varepsilon\|_{H^1(\Omega)}),
\end{aligned}
\end{equation}
for any $t> 0$. Then we finish the proof.
\end{proof}

The following proposition establishes the uniform estimate for $p^\varepsilon$, which will be crucial in proving the convergence of solutions.
\begin{proposition}\label{uniform-p}
Let $(\Bu^\varepsilon,\psi^\varepsilon,p^\varepsilon)$ be the solution to the problem \eqref{Sharps}-\eqref{bc} with permeability $K^\varepsilon$. Then there exists a constant $C$ independent of $\varepsilon$, such that
\begin{equation}\label{uniform-p-1} 
  \|p^\varepsilon\|_{H^1(\Omega)}   \leq   C\|\psi^\varepsilon\|_{L^2(\Omega)}.
\end{equation}  
Furthermore, we have also
\begin{equation}\label{uniform-p-1-1} 
  \|\partial_xp^\varepsilon\|_{H^1(\Omega)}   \leq   C\|\partial_x\psi^\varepsilon\|_{L^2(\Omega)}.
\end{equation} 
\end{proposition}
\begin{proof} 
We write the $x$-periodic function $p^\varepsilon$ as 
\begin{equation}\nonumber
p^\varepsilon(x,z) = p_0^\varepsilon(z) +  \sum_{n\in \mathbb{Z},n\neq 0} p^\varepsilon_n(z) e^{\ii n \frac{2\pi}{L} x}=: p_0^\varepsilon(z) + p_{\neq}^\varepsilon(x,z),
\end{equation}
where 
\begin{equation}\nonumber
p^\varepsilon_n (z) = \frac{1}{L}\int_{0}^{L} p^\varepsilon(x,z) e^{- \ii n \frac{2\pi}{L}x} \dd x. 
\end{equation}
Similarly, we have 
\begin{equation}\nonumber
\psi^\varepsilon(x,z) = \psi^\varepsilon_0(z) + \psi^\varepsilon_{\neq}(x,z)
\end{equation}
and 
\begin{equation}\nonumber
  \Bu^\varepsilon = \Bu^\varepsilon_0(z) + \Bu^\varepsilon_{\neq}(x,z)= u^\varepsilon_{x,0}(z) \Be_x+ u^\varepsilon_{z,0}(z)\Be_z+ u^\varepsilon_{x,\neq}(x,z) \Be_x + u^\varepsilon_{z,\neq}(x,z) \Be_z .
\end{equation}

For the zero mode, it follows from \eqref{pressure} that $p^\varepsilon_0$ satisfies
\begin{equation}\nonumber
\left\{\begin{aligned}
&-(K^\varepsilon(p^\varepsilon_{0})')' = (K^\varepsilon\psi^\varepsilon_0)',\quad \text{ in }(-H,0)\\
&(p^\varepsilon_0)'(0)= (p^\varepsilon_0)'(-H)= 0.
\end{aligned}\right.
\end{equation}
Noting that $K^\varepsilon$ is constant in each interval $(z_{i-1},z_i)$, then $K^\varepsilon((p^\varepsilon_0)'+ \psi^\varepsilon_0)$ is piecewise constant. Since $K^\varepsilon((p^\varepsilon_0)'+ \psi^\varepsilon_0)$ is continuous at any $z= z_{i}$, and vanishes at the boundary $z=0,-H$, we conclude that 
\begin{equation}\label{uniform-p-2} 
(p^\varepsilon_0)'+  \psi^\varepsilon_0=0.
\end{equation}
On the other hand, since $p^\varepsilon$ has zero mean value, one has 
\begin{equation}\label{uniform-p-2-1} 
0 = \int_\Omega p^\varepsilon \dd x\dd z  = L\int_{-H}^0 p^\varepsilon_0\dd z.
\end{equation}
Then an application of Poincar\'e inequality gives 
\begin{equation}\label{uniform-p-3} 
\| p^\varepsilon_0 \|_{H^1(\Omega)} \leq C\| p^\varepsilon_0 \|_{H^1(-H,0)}  \leq C \|\psi^\varepsilon_0\|_{L^2(-H,0)} \leq C \|\psi^\varepsilon_0\|_{L^2(\Omega)}.
\end{equation}

Furthermore, according to \eqref{pressure} and \eqref{Sharps}$_1$, the nonzero mode $p^\varepsilon_{\neq}$ satisfies  
\begin{equation}  \nonumber
\left\{\begin{aligned}
&-\nabla\cdot(K^\varepsilon p^\varepsilon_{\neq} ) = \nabla \cdot(K^\varepsilon \psi^\varepsilon_{\neq})   , ~~~~&& \text{ in } \Omega_\pm \\
&  \partial_z p^\varepsilon_{\neq}= -\psi^\varepsilon_{\neq}  -\frac{1}{K^\varepsilon} u^\varepsilon_{z,\neq},  &&\text{ on } \partial\Omega_\pm.
  \end{aligned}\right. 
\end{equation}  
Noting that 
\begin{equation}\nonumber
\int_{0}^{L} p^\varepsilon_{\neq}(x,z) \dd x = 0,\quad \text{  for any }z\in (-H,0),
\end{equation}
one has 
\begin{equation}\nonumber
\int_{\Omega_\pm} p^\varepsilon_{\neq}  \dd x\dd z = 0.
\end{equation}
Then it follows from the standard elliptic estimates and \eqref{H^1-7-1} that  
\begin{equation}\label{uniform-p-4}
\begin{aligned}
\| p^\varepsilon_{\neq}\|_{H^1(\Omega_\pm)} \leq &~C(\|\psi^\varepsilon_{\neq}\|_{L^2(\Omega\pm)} +\|u^\varepsilon_{z,\neq}\|_{H^{-\frac12} (\partial\Omega_\pm)})\\
\leq&~ C(\|\psi^\varepsilon_{\neq}\|_{L^2(\Omega_\pm)}+\|\Bu^\varepsilon_{\neq}\|_{L^2(\Omega_0)}) \\
\leq&~ C \|\psi^\varepsilon \|_{L^2(\Omega)}.
\end{aligned}
\end{equation}
since $K^\varepsilon$ admits a positive lower bound in $\Omega_\pm$, which is independent of $\varepsilon$.

Meanwhile, we have also 
\begin{equation}\nonumber
  -\Delta p^\varepsilon_{\neq } = \partial_z\psi^\varepsilon_{\neq}     ,\quad \text{ in }\Omega_0.
\end{equation}  
Since $p^\varepsilon_{\neq}$ is continuous on the interfaces, it follows from the $H^1$ estimates for Dirichlet problem (\cite[Theorem 8.3]{GT}) and \eqref{uniform-p-4} that 
\begin{equation}\label{uniform-p-5} 
\begin{aligned}
\|p^\varepsilon_{\neq}\|_{H^1(\Omega_0)} \leq&~ C(\|\psi^\varepsilon_{\neq}\|_{L^2(\Omega_0)} + \|p^\varepsilon_{\neq}\|_{H^\frac12(\partial\Omega_+\cap \partial\Omega_0)} + \|p^\varepsilon_{\neq}\|_{H^\frac12 (\partial\Omega_-\cap\partial\Omega_0)}) \\
\leq&~ C(\|\psi^\varepsilon_{\neq}\|_{L^2(\Omega_0)} + \|p^\varepsilon_{\neq}\|_{H^1(\Omega_+)} + \|p^\varepsilon_{\neq}\|_{H^1 (\Omega_-)})\\
\leq&~ C \|\psi^\varepsilon\|_{L^2(\Omega)}. 
\end{aligned}
\end{equation}
Combining \eqref{uniform-p-3} and \eqref{uniform-p-5}, one obtains \eqref{uniform-p-1}. The estimate \eqref{uniform-p-1-1} can be proved in a same way. Then we complete the proof of this proposition.

\end{proof}

\section{Convergence} 
This section is devoted to the convergence analysis of solutions as the permeability parameter $\varepsilon$ tends to zero. We establish two main results: first, a quantitative estimate for the convergence of solutions over arbitrary finite time intervals is proved in Theorem \ref{convergence}; second, we investigate the limiting behavior of the associated infinite-dimensional dynamical systems in Subsection \ref{convergence-attractor}. It is shown that for any $\varepsilon\ge 0$, the semigroup associated with the problem \eqref{Sharps}-\eqref{bc} possesses a global attractor in $H$, and the family of attractors converges to the attractor of the limit system in the Hausdorff semi-distance. The argument combines the finite-time convergence result with the standard theory of attractors for dissipative systems.

\subsection{Finite-time convergence}
\begin{theorem}\label{convergence}
For any $\varepsilon\ge 0$, let $(\Bu^\varepsilon,\psi^\varepsilon,p^\varepsilon)$ be a solution to the problem \eqref{Sharps}-\eqref{bc} with permeability $K^\varepsilon$, where the pressure $p^0$ is constructed as in Lemma \ref{degenerate-p} if $\varepsilon=0$. If the initial data satisfies $ \psi_{in}^\varepsilon= \psi_{in}^0 \in V$ for any $\varepsilon\ge 0$, then one has 
  \begin{equation}\label{convergence-0}
\|(\Bu^\varepsilon -\Bu^0, \psi^\varepsilon -\psi^0)(t)\|_{L^2(\Omega)}+ \|(p^\varepsilon -p^0)(t)\|_{H^1(\Omega)} \leq M_7 \varepsilon ,\quad\text{ for any }t\in [0,T],
  \end{equation}
where the constant $M_7 =M_7(T, \|\psi_{in}\|_{H^1(\Omega)})$ is independent of $\varepsilon$.  
\end{theorem}

\begin{proof} The proof is divided into four steps.

{\em Step 1. Set up.}
Denote 
\[(\tilde{\Bu},\tilde{\psi},\tilde{p}):= (\Bu^\varepsilon,\psi^\varepsilon,p^\varepsilon) - (\Bu^0,\psi^0,p^0)\]
and 
\[ \tilde{K} := K^\varepsilon - K^0.\]
Then the difference $(\tilde{\Bu},\tilde{\psi},\tilde{p})$ satisfies the following system
\begin{equation}\label{error} 
      \left\{\begin{aligned}
        &\tilde{\Bu}=- K ^\varepsilon\left(\nabla \tilde{p } + \tilde{\psi} \Be_{z}\right)- \tilde{K} \left(\nabla p^0  + \psi ^0\Be_{z}\right),\\
        &\nabla \cdot \tilde{\Bu} =0,\\
        & \partial_t \tilde{\psi} + \Bu^0\cdot\nabla \tilde{\psi}+\tilde{\Bu}\cdot\nabla \psi^\varepsilon  +  \phi_b'\tilde{\Bu} \cdot \Be_z   - \nabla \cdot(D \nabla\tilde{\psi}) =0.
      \end{aligned}\right.  
  \end{equation}   
{\em Step 2. Estimate of $\tilde{\Bu}$.} Noting that $\Bu^0$ is identically zero in $\Omega_0$, the difference $\tilde{\Bu}$ satisfies
\begin{equation}\nonumber
\tilde{\Bu} = \Bu^\varepsilon = - \varepsilon \left(\nabla p^\varepsilon + \psi^\varepsilon \Be_{z}\right),\quad \text{ in }\Omega_0 .
\end{equation}  
According to Proposition \ref{uniform-p} and Poincar\'e inequality, one has 
\begin{equation}\label{convergence-1}
\|\Bu^\varepsilon\|_{L^2(\Omega_0)}  \leq  C\varepsilon(\|\nabla p^\varepsilon\|_{L^2(\Omega_0)}
 + \|\psi^\varepsilon\|_{L^2(\Omega_0)})  \leq C\varepsilon \|\nabla\psi^\varepsilon\|_{L^2(\Omega)}.
\end{equation}

Recall that the pressure $p^\varepsilon$ is piecewise $H^2$ and satisfies
\[ -\Delta p^\varepsilon = \partial_z\psi^\varepsilon,\quad \text{ in }\Omega_0.\]
Then we use Proposition \ref{uniform-p} to deduce 
\begin{equation}\nonumber
  \|\partial_{z}^2 p^\varepsilon\|_{L^2(\Omega_0) }\leq   \|\partial_{x}^2 p^\varepsilon\|_{L^2(\Omega_0) }+ \|\partial_{z}\psi^\varepsilon\|_{L^2(\Omega_0) } 
  \leq 2\|\nabla\psi^\varepsilon\|_{L^2(\Omega) }.
\end{equation}
Utilizing \eqref{Sharps}$_1$ and Proposition \ref{uniform-p} again, we have  
\begin{equation}\label{convergence-1-2}
\begin{aligned}
  \|\nabla\Bu^\varepsilon\|_{L^2(\Omega_0)} \leq&~
   \varepsilon(\| \nabla^2 p^\varepsilon\|_{L^2(\Omega_0)} +\| \nabla \psi^\varepsilon\|_{L^2(\Omega_0)})\\ 
  \leq&~ \varepsilon (\|\partial_x \nabla p^\varepsilon\|_{L^2(\Omega_0)} + \|\partial_z^2 p^\varepsilon\|_{L^2(\Omega_0)} +\|\nabla \psi^\varepsilon\|_{L^2(\Omega_0)})\\
  \leq&~C \varepsilon\|\nabla \psi^\varepsilon\|_{L^2(\Omega)}.
\end{aligned}
\end{equation} 
By virtue of Gagliardo-Nirenberg inequality, Poincar\'e inequality and \eqref{convergence-1}-\eqref{convergence-1-2}, one has
\begin{equation}\label{convergence-2} 
\begin{aligned}
\|\Bu^\varepsilon\|_{L^4(\Omega_0)}\leq &~  C(\|\Bu^\varepsilon\|_{L^2(\Omega_0)}^\frac12\|\nabla\Bu^\varepsilon\|_{L^2(\Omega_0)}^\frac12 + \|\Bu^\varepsilon\|_{L^2(\Omega_0)})  \\
 \leq&~C \varepsilon  ( \| \psi^\varepsilon\|_{L^2(\Omega)}^\frac12\| \nabla\psi^\varepsilon\|_{L^2(\Omega)}^\frac12 + \| \psi^\varepsilon\|_{L^2(\Omega)})\\
\leq&~ C \varepsilon \|\nabla\psi^\varepsilon\|_{L^2(\Omega)}. 
\end{aligned}
\end{equation}  
On the other hand, since  $K^\varepsilon = K^0$ in $\Omega_{\pm}$, the difference $\tilde{\Bu}$ satisfies 
\begin{equation}\nonumber
\tilde{\Bu}   = - K^0\left(\nabla \tilde{p} + \tilde{\psi} \Be_{z}\right),\quad \text{ in }\Omega_\pm,
\end{equation} 
where $\tilde{p}$ is determined by the following elliptic problem
\begin{equation} \nonumber
  \left\{\begin{aligned}
&-\nabla \cdot(K^0 \nabla \tilde{p}) = \nabla \cdot(K^0 \tilde{\psi}\Be_{z}) , ~~~~ &&\text{ in } \Omega_\pm \\
& \partial_z \tilde{p}= -\frac{1}{K^0} \tilde{u}_z- \tilde{\psi} \Be_z ,&&\text{ on } \partial \Omega_\pm. 
  \end{aligned}\right. 
\end{equation}
Noting that $\tilde{u}_z =0$ on $\partial \Omega$ and $\tilde{u}_z=u^\varepsilon_z$ on $\partial\Omega_0$,  it follows from the $L^r$ estimate for the elliptic problem with small partial BMO coefficients (\cite{DL2021}) that
\begin{equation}\nonumber
\begin{aligned}
\|\nabla \tilde{p} \|_{L^r(\Omega_\pm)} \leq &~ C\|\tilde{\psi} \|_{L^r(\Omega_\pm)} +C\|\tilde{u}_z\|_{W^{-\frac{1}{r},r}(\partial\Omega_\pm)} \\
\leq&~  C\|\tilde{\psi} \|_{L^r(\Omega_\pm)} +C\|u^\varepsilon_z\|_{W^{-\frac{1}{r},r}(\partial\Omega_0)} \\
\leq&~ C\|\tilde{\psi} \|_{L^r(\Omega_\pm)} +C\|\Bu^\varepsilon\|_{L^r(\Omega_0)},
\end{aligned}
\end{equation}
for any $r\in (1,\infty)$. Then we have 
\begin{equation}\nonumber
\|\tilde{\Bu} \|_{L^r(\Omega_\pm)} \leq C(\|\nabla \tilde{ p} \|_{L^r(\Omega_\pm)}+ \|\tilde{\psi} \|_{L^r(\Omega_\pm)}) \leq C\|\tilde{\psi} \|_{L^r(\Omega_\pm)} +C\|\Bu^\varepsilon\|_{L^r(\Omega_0)}.
\end{equation} 
This, together with \eqref{convergence-1} and \eqref{convergence-2}, gives
\begin{equation}\label{convergence-3}
\|\tilde{\Bu} \|_{L^r(\Omega)} \leq  C\|\tilde{\psi} \|_{L^r(\Omega)} +C\|\Bu^\varepsilon\|_{L^r(\Omega_0)} \leq C\|\tilde{\psi} \|_{L^r(\Omega)}+C \varepsilon \|\nabla\psi^\varepsilon\|_{L^2(\Omega)},
\end{equation} 
for $r=2,4$.

{\em Step 3. Estimate of $\tilde{\psi}$.} Now we test the equation \eqref{Sharps}$_1$ by $\tilde{\psi}$ and use integration by parts to obtain
\begin{equation}\label{convergence-7}  
 \frac12\frac{\dd }{\dd t}\|\tilde{\psi}\|_{L^2(\Omega)}^2  + \|\sqrt{D } \nabla \tilde{\psi}\|_{L^2(\Omega)}^2= (\tilde{\Bu}\cdot\nabla   \tilde{\psi}, \psi^\varepsilon) -  (\phi_b'\tilde{\Bu} \cdot \Be_z, \tilde{\psi}) . 
\end{equation} 
With the aid of Young's inequality, Poincar\'e inequality, Lemma \ref{lemmaA1} and \eqref{convergence-3}, one has
\begin{equation}\label{convergence-4}
\begin{aligned}
&~|(\tilde{\Bu}\cdot\nabla   \tilde{\psi}, \psi^\varepsilon)| \\
\leq&~ \|\nabla\tilde{\psi}\|_{L^2(\Omega)}\|\tilde{\Bu}\|_{L^4(\Omega)}\|\psi^\varepsilon\|_{L^4(\Omega)}\\
\leq &~C\|\nabla\tilde{\psi}\|_{L^2(\Omega)}(\|\tilde{\psi} \|_{L^4(\Omega )}+  \varepsilon \|\nabla\psi^\varepsilon\|_{L^2(\Omega)})  \|\nabla\psi^\varepsilon\|_{L^2(\Omega)}^\frac12\|\psi^\varepsilon\|_{L^2(\Omega)}^\frac12 \\
\leq &~C\|\nabla\tilde{\psi}\|_{L^2(\Omega)}(\|\nabla \tilde{\psi} \|_{L^2(\Omega )}^\frac12\|\tilde{\psi} \|_{L^2(\Omega )}^\frac12+  \varepsilon \|\nabla\psi^\varepsilon\|_{L^2(\Omega)})  \|\nabla\psi^\varepsilon\|_{L^2(\Omega)}\\
\leq &~C \|\nabla \tilde{\psi} \|_{L^2(\Omega )}^\frac32\|\tilde{\psi} \|_{L^2(\Omega )}^\frac12\|\nabla\psi^\varepsilon\|_{L^2(\Omega)} +C \varepsilon \|\nabla\psi^\varepsilon\|_{L^2(\Omega)}^2\|\nabla\tilde{\psi}\|_{L^2(\Omega)}  \\ 
\leq&~\frac{\min_i D_i}{2} \|\nabla  \tilde{\psi}\|_{L^2(\Omega)}^2 + C\|\tilde{\psi} \|_{L^2(\Omega )}^2\|\nabla\psi^\varepsilon\|_{L^2(\Omega)}^4 +  C\varepsilon^2\|\nabla\psi^\varepsilon\|_{L^2(\Omega)}^4,
\end{aligned}
\end{equation}
and 
\begin{equation}\label{convergence-5}
\begin{aligned}
|( \phi_b'\tilde{\Bu} \cdot \Be_z, \tilde{\psi})|  \leq&~ c_\Delta\delta^{-1} \|\tilde{\Bu}\|_{L^2(\Omega)}\|\tilde{\psi}\|_{L^2(\Omega)}\\
\leq&~C \| \tilde{\psi}\|_{L^2(\Omega)}(\|\tilde{\psi} \|_{L^2(\Omega )}+\varepsilon \|\nabla\psi^\varepsilon\|_{L^2(\Omega)})\\
\leq&~C \| \tilde{\psi}\|_{L^2(\Omega)}^2+ C \varepsilon^2 \|\nabla\psi^\varepsilon\|_{L^2(\Omega)}^2.
\end{aligned}
\end{equation}  
Substituting \eqref{convergence-4}-\eqref{convergence-5} to \eqref{convergence-7} and using Proposition \ref{H^1}, we have
\begin{equation} \nonumber 
\begin{aligned}
&~\frac{1}{2}\frac{\dd }{\dd t} \|\tilde{\psi}\|_{L^2(\Omega)}^2 + \frac{\min_i D_i}{2}\|\nabla \tilde{\psi}\|_{L^2(\Omega)}^2  \\
\leq& ~ C (1+ \|\nabla\psi^\varepsilon\|_{L^2(\Omega)}^4)\|\tilde{\psi}\|_{L^2(\Omega)}^2+ C\varepsilon^2( \|\nabla \psi^\varepsilon\|_{L^2(\Omega)}^4 +   \|\nabla \psi^\varepsilon\|_{L^2(\Omega)}^2)\\
\leq& ~ C (1+ M_6^4)\|\tilde{\psi}\|_{L^2(\Omega)}^2+ C ( M_6^4 +   M_6^2) \varepsilon^2.
\end{aligned}
\end{equation}
Since $\tilde{\psi}(0)=0$, one utilizes Gr\"onwall's inequality to conclude
\begin{equation}\label{convergence-8} 
\begin{aligned}
\|\tilde{\psi}(t)\|_{L^2(\Omega)}^2 \leq& ~  C ( M_6^4 +   M_6^2)\varepsilon^2  \int_0^t   e^{  C (1+ M_6^4)(t-\tau)} \dd \tau \\
\leq&~ C ( M_6^4 +   M_6^2)e^{C (1+ M_6^4)t} \varepsilon^2 .
\end{aligned}
\end{equation}
This, together with \eqref{convergence-3} and Proposition \ref{H^1}, yields
\begin{equation}\label{convergence-9}
  \|\tilde{\Bu}(t) \|_{L^2(\Omega)} \leq C\|\tilde{\psi}(t)\|_{L^2(\Omega)} + CM_6 \varepsilon \leq C ( M_6^2 +   M_6 )e^{C (1+ M_6^4)t} \varepsilon.
\end{equation}

{\em Step 4. Estimate of $\tilde{p}$.}
For the convergence of $p^\varepsilon$, similar to the proof of Proposition \ref{uniform-p}, we consider the Fourier mode of $\tilde{p}$. For the zero mode, by virtue of \eqref{degenerate-p-0} and \eqref{uniform-p-2}-\eqref{uniform-p-2-1}, one has 
\begin{equation}\label{convergence-10}
 \|\tilde{p}_0\|_{H^1(\Omega)} = \| p _0-  p^\varepsilon_0  \|_{H^1(\Omega)} \leq  C \| \tilde{\psi}_0 \|_{L^2(\Omega)} .
\end{equation}
On the other hand, the nonzero mode $\tilde{p}_{\neq}$ satisfies
\begin{equation}  \nonumber
\left\{\begin{aligned}
&-\nabla\cdot(K^0 \nabla \tilde{p}_{\neq} ) = \nabla \cdot(K^0 \tilde{\psi}_{\neq})   , ~~~~&& \text{ in } \Omega_\pm \\
&  \partial_z \tilde{p}_{\neq}= -\tilde{\psi}_{\neq}  -\frac{1}{K^0} \tilde{u}_{z,\neq},  &&\text{ on } \partial\Omega_\pm 
  \end{aligned}\right. 
\end{equation}  
and 
\begin{equation}\nonumber
\int_{\Omega_\pm} \tilde{p}_{\neq}  \dd x\dd z  = 0.
\end{equation}
Then it follows from the standard elliptic estimates  and \eqref{convergence-3} that  
\begin{equation}\label{convergence-11}
\begin{aligned}
\| \tilde{p}_{\neq}\|_{H^1(\Omega_\pm)} \leq &~C(\|\tilde{\psi}_{\neq}\|_{L^2(\Omega\pm)} +\|\tilde{u}_{z,\neq}\|_{H^{-\frac12} (\partial \Omega_\pm)})\\
\leq&~ C(\|\tilde{\psi}_{\neq}\|_{L^2(\Omega_\pm)}+\|\tilde{\Bu}_{\neq}\|_{L^2(\Omega_\pm)}) \\
\leq &~C \|\tilde{\psi} \|_{L^2(\Omega)} + C\|\tilde{\Bu}\|_{L^2(\Omega)}.
\end{aligned}
\end{equation} 
In $\Omega_0$,  we have that
\begin{equation}\nonumber
  -\Delta \tilde{p}_{\neq } = \partial_z\tilde{\psi}_{\neq}   .
\end{equation}  
Since $\tilde{p}_{\neq}$ is continuous on the interfaces, it follows from the $H^1$ estimates for Dirichlet boundary problem and \eqref{convergence-11} that 
\begin{equation}\label{convergence-12} 
\begin{aligned}
\|\tilde{p}_{\neq}\|_{H^1(\Omega_0)} \leq&~ C(\|\tilde{\psi}_{\neq}\|_{L^2(\Omega_0)} + \|\tilde{p}_{\neq}\|_{H^\frac12(\partial\Omega_+\cap \partial\Omega_0)} + \|\tilde{p}_{\neq}\|_{H^\frac12 (\partial\Omega_-\cap\partial\Omega_0)}) \\
\leq&~ C(\|\tilde{\psi}_{\neq}\|_{L^2(\Omega_0)} + \|\tilde{p}_{\neq}\|_{H^1(\Omega_+)} + \|\tilde{p}_{\neq}\|_{H^1 (\Omega_-)})\\
\leq&~ C \|\tilde{\psi} \|_{L^2(\Omega)} + C\|\tilde{\Bu}\|_{L^2(\Omega)}. 
\end{aligned}
\end{equation}
 Combining \eqref{convergence-10}-\eqref{convergence-12}, one has 
\begin{equation}\nonumber
  \|\tilde{p}\|_{H^1(\Omega)} \leq C \|\tilde{\psi} \|_{L^2(\Omega)} + C\|\tilde{\Bu}\|_{L^2(\Omega)} \leq C ( M_6^2 +   M_6 )e^{C (1+ M_6^4)t} \varepsilon.
\end{equation}
This, together with \eqref{convergence-8}-\eqref{convergence-9}, gives \eqref{convergence-0}. Then we complete the proof this theorem.
\end{proof}

\subsection{Convergence of attractor}\label{convergence-attractor} Now we investigate the long-time behavior of the sharp interface model with vanishing permeability. The proof follows a standard argument, combining the attracting property of the global attractor of the limit model with the uniform convergence (with respect to initial data taken from the global attractor) of trajectories on finite time intervals, as established in the preceding subsection.

We first define the semigroup  associated with problem \eqref{Sharps}-\eqref{bc}. 
For any $\varepsilon\ge 0$ and $t$,  set 
\[S_\varepsilon (t):\psi^\varepsilon_{in}\mapsto \psi^\varepsilon(t),\]
where $\psi^\varepsilon$ is the solution to the problem \eqref{Sharps}-\eqref{bc} with permeability $K^\varepsilon$.

By Proposition \ref{prop1} and \ref{existence-degenerate}, $S_\varepsilon(t)$ is continuous from $H$ to itself, and is continuous in $t$, for any $\varepsilon\ge 0$. Consequently, we arrive at the following theorem on the existence and convergence of the attractors  corresponding to the semigroup $\{S_\varepsilon(t)\}$.

\begin{theorem}\label{thm1}
  For any $\varepsilon\ge  0$, there exists a global attractor  $\mathcal{A}_\varepsilon \subset H$ for the semigroup $\{S_\varepsilon(t)\}$ associated with the problem \eqref{Sharps}-\eqref{bc} with permeability $K^\varepsilon$, which is uniformly bounded in $V$, compact and connected in $H$. Furthermore, as $\varepsilon$ goes to $0^+$, the attractor $\mathcal{A}_\varepsilon$ converges to  $\mathcal{A}_0$ in the following sense,
\begin{equation}\nonumber
d(\mathcal{A}_{\varepsilon}, \mathcal{A}_0) \to  0,
\end{equation}
where $d(\mathcal{A}_{\varepsilon}, \mathcal{A}_0)$ is the Hausdorff semi-distance between $\mathcal{A}_{\varepsilon}, \mathcal{A}_0$ defined by 
\begin{equation}\nonumber
d(\mathcal{A}_{\varepsilon}, \mathcal{A}_0) := \sup_{f\in \mathcal{A}_{\varepsilon}} \operatorname{dist} (f,\mathcal{A})= \sup_{f\in \mathcal{A}_{\varepsilon}} \inf_{g\in \mathcal{A}_0}\|f-g\|_{L^2(\Omega)}.
\end{equation}
\end{theorem}
\begin{proof}
Proposition \ref{L^2}, together with Proposition \ref{H^1}, implies that 
\begin{equation}\nonumber
  B_1 =\left\{\psi\in H:~ \|\psi\|_{L^2(\Omega)}^2 < \frac{M_1H^2}{\min_i D_i}+1 ,~\|\sqrt{D} \nabla\psi\|_{L^2(\Omega)}^2 \leq M_5\right\}
\end{equation}
is an absorbing set for the semigroup $\{S_\varepsilon(t)\}$ for any $\varepsilon\ge 0$. More precisely, if $\psi_{in}^\varepsilon \in B(0, R) \subset H$, then $\psi^\varepsilon(t) = S_\varepsilon(t) \psi_{in}^\varepsilon$ enters the absorbing set $B_1$ for some $t \leq T_1(R) + 1$ and stays in it for all $t \geq T_1(R) + 1$, where $T_1(R) = \frac{2H^2}{\min_i D_i}\ln R$. Furthermore, the absorbing set $B_1$ is bounded in $V$ and hence, relatively compact in $H$. This ensures the uniformly compactness of $S_\varepsilon(t)$ for large $t$. Then, according to \cite[Theorem 1.1]{Temam1997}, the semigroup $\{S_\varepsilon(t)\}$ associated with problem \eqref{Sharps}-\eqref{bc} possesses a global attractor $\mathcal{A}_\varepsilon$, which is defined as the $\omega$-limits of $B_1$: 
\begin{equation}\nonumber
\mathcal{A}_\varepsilon = \bigcap_{s \geq 0} \overline{\bigcup_{t \geq s} S_\varepsilon(t) B_1}.
\end{equation}
In particular, $\mathcal{A}_\varepsilon$ is  compact and connected in $H$, as well as bounded in $V$.  

To show the convergence, we assume for contradiction that the assertion is false. Then there exists a  constant $c >0$ and a sequence $\{\varepsilon_k\}$ converging to $0^+$ such that
\[ d(\mathcal{A}_{\varepsilon_k},\mathcal{A}_0) \geq c ,~~~~\text{ for all } k \geq 1.
\]
Since the global attractor $\mathcal{A}_{\varepsilon_k}$ are compact in $H$, there exists an $a_k \in \mathcal{A}_{\varepsilon_k}$ such that
\[\operatorname{dist} (a_k,\mathcal{A}_0) =d(\mathcal{A}_{\varepsilon_k},\mathcal{A}_0) \geq c.\]
On the other hand, since $\mathcal{A}_0$ absorbs all the bounded sets in $H$, we choose $T > 0$ to be sufficiently large such that
\[d(S_0(T)B_1 ,\mathcal{A}_0)<\frac{c}{4}.\]
Since the set $\mathcal{A}_\varepsilon$ is invariant with respect to $S_\varepsilon(t)$ for any $t\ge 0$, i.e., $S_\varepsilon (t) \mathcal{A}_\varepsilon= \mathcal{A}_\varepsilon$, then for each $a_k \in \mathcal{A}_{\varepsilon_k}$, there exists $b_k \in \mathcal{A}_{\varepsilon_k}$ such that $S_{\varepsilon_k}(T)b_k= a_k$ for all $k$. This implies that $b_k\in \mathcal{A}_{\varepsilon_k} \subset B_1$ for all $k$. By Proposition \ref{convergence}, we have
\begin{equation}\nonumber
\begin{aligned}
  \|a_k -S_0(T)b_k\|_{L^2(\Omega)} \leq&~ \|S_{\varepsilon_k}(T)b_k-S_0(T)b_k\|_{L^2(\Omega)}\\
  \leq&~ \varepsilon_k  M_7\left(T,\sqrt{ \frac{M_1H^2}{\min_i D_i}+1+  \frac{M_5}{\min_i D_i}}\right)\\
  \leq&~\frac{c}{4},
\end{aligned}
\end{equation}
provided $k$ is sufficiently large. Hence,
\begin{equation}\nonumber
\begin{aligned}
  \operatorname{dist}(a_k, \mathcal{A}_0) \leq &~\|a_k -S_0(T)b_k\|_{L^2(\Omega)} + \operatorname{dist}(S_0(T)b_k, \mathcal{A}_0)\\
  \leq&~  \|a_k -S_0(T)b_k\|_{L^2(\Omega)} + d(S_0(T)B_1, \mathcal{A}_0) \\
  \leq&~ \frac{c}{2}.
\end{aligned}
\end{equation}
This contradiction completes the proof of this theorem.
\end{proof}

\section{Conclusion}

We have investigated the singular limit in which the permeability of one layer tends to zero for the two-dimensional Darcy--Boussinesq system posed in a multilayer porous medium. This limit is singular in that the pressure equation in the vanishing-permeability layer becomes degenerate, leading to a reduced system with an inherent structural singularity. The well-posedness of the limit problem has been established through an appropriate decomposition of the solutions with respect to the individual layers and their horizontal frequency components. Convergence to the limit system is proved both on arbitrary finite-time intervals in the $L^{2}$ norm and at the level of the global attractors.

While we treated the straight interface only in the work, the case of curved interface can be studies analogously with the help of curvilinear coordinates as long as the interfaces are well separated. The case of solution dependent diffusivity is a much more challenging problem.
Based on an explicit solution of a closely related elliptic problem, we speculate that the convergence rate obtained in the analysis is optimal and cannot, in general, be improved. The velocity field in the zero-permeability layer is shown to be identically zero, which effectively inhibits global convection. The influence of such low-permeability barriers on long-time statistical quantities, including the time-averaged Nusselt number, remains an important open problem. Motivated by practical settings in which certain layers are simultaneously thin and of low permeability, as in \cite{hewitt2022jfm, hewitt2020jfm}, the combined limits of vanishing permeability and vanishing layer thickness for one or more layers would be of considerable interest.

Furthermore, many applications, such as CO$_2$ sequestration, involve multiphase flows in heterogeneous porous formations. Extending the present results to thermodynamically consistent phase-field models for fluid--porous systems, and to configurations with permeability contrasts across multiple thin layers, represents a natural next step. Finally, similar questions arise for coupled fluid-porous or geophysical convection models in three dimensions. These directions, together with the potential incorporation of multiphysics effects relevant to environmental and industrial applications, all merit additional investigation in future work.

\appendix
\section{An elliptic example with explicit convergence rate}\label{appendixA} 
In this appendix, we present a simple explicit calculation for the linear elliptic problem \eqref{pressure}, illustrating the convergence rate of solutions as the permeability of the intermediate layer tends to zero. Consider the elliptic problem 
\begin{equation}\label{A-1}
\left\{\begin{aligned}
&-\nabla\cdot(K^\varepsilon \nabla p^\varepsilon) =  \nabla\cdot(K^\varepsilon \psi\Be_z),&& \text{ in } \Omega\\
&\partial_z p^\varepsilon+ \psi=0,&&\text{ on }\partial\Omega,
\end{aligned}\right.
\end{equation}
where the domain $\Omega=(0,2\pi)\times (-2,2)$ is divided into three layers with two interfaces located at $z=\pm1$, and 
\begin{equation}\label{}
K^\varepsilon(x,z) =\left\{\begin{aligned}
 & 1,\quad&& \text{ if }z\in (1,2),\\
 & \varepsilon,\quad&& \text{ if } z\in (-1,1),\\
 & 1,\quad&& \text{ if } z\in (-2,-1).
\end{aligned}\right.
\end{equation} 
Assume that
\begin{equation}\nonumber
\psi=   \psi_1(z) e^{\ii x}  \quad \text{ and }\quad p^\varepsilon = p^\varepsilon_1(z) e^{\ii x}.
\end{equation}
Then \eqref{A-1} reduces to the ODE problem 
\begin{equation} \label{A-2}
\left\{\begin{aligned}
&-  (K^\varepsilon (p^\varepsilon_1)')' + K^\varepsilon  p^\varepsilon_1=  (K^\varepsilon \psi_1)',\quad \text{ in }(-2,2),\\
& ((p^\varepsilon_1)'+\psi_1)|_{z=\pm 2} = 0.
\end{aligned}\right.
\end{equation}
Suppose that $\psi_1$ is a continuous, piecewise linear function such that $\psi_1'= c_+,c_0,c_-$ in the intervals $(1,2), (-1,1), (-2,-1)$, respectively. Then the solution to the problem \eqref{A-2} admits the representation
\begin{equation}\nonumber
p^\varepsilon_1  = \left\{\begin{aligned}
 & a_+ e^z + b_+ e^{-z} + c_+ ,\quad&& \text{ if }z\in (1,2),\\
 & a_0 e^z + b_0 e^{-z} + c_0 ,\quad&& \text{ if } z\in (-1,1),\\
 & a_- e^z + b_- e^{-z} + c_- ,\quad&& \text{ if } z\in (-2,-1).
\end{aligned}\right.
\end{equation} 
By imposing the continuity of $p^\varepsilon$ and of the flux $K^\varepsilon(\partial_z p^\varepsilon+\psi)$ across each interface, together with the boundary conditions on $\partial\Omega$, the coefficients satisfy the following system:
\begin{equation}\label{A-3}
\left\{\begin{aligned} 
  &a_+ e^2 - b_+ e^{-2} + \psi_1(2)=0, \\
 &a_- e^{-2} -b_-e^{2}+\psi_1(-2)=0,\\
 &a_+ e  + b_+ e^{-1}  = a_0 e + b_0 e^{-1},\\
 &a_0 e^{-1} + b_0 e  = a_- e^{-1} + b_- e,\\ 
& a_+ e - b_+ e^{-1} + \psi_1(1) = \varepsilon (a_0 e  - b_0 e^{-1} + \psi_1(1)),\\
& \varepsilon(a_0 e^{-1} - b_0 e  +\psi_1(-1)) =  a_- e^{-1} - b_- e  +\psi_1(-1).
\end{aligned}\right.
\end{equation}
Hence the coefficients $a_i,b_i$ are uniquely determined by the values $\psi_1(\pm1)$ and $\psi_1(\pm2)$.

As a concrete example, take
$\psi(2)=0,\psi(1)=-1,\psi(-1)=1,\psi(-2)=0$. Solving the linear system \eqref{A-3} gives
\begin{equation}\nonumber
\left\{\begin{aligned} 
&a_+ = b_-  \frac{\varepsilon-1}{e (e^2-1)(\varepsilon+1)}, \\ 
&a_0 =b_0 \frac{e (\varepsilon-1)}{(e^2-1)(\varepsilon+1)}, \\ 
&a_- = b_+ =  \frac{e^3(\varepsilon-1)}{(e^2-1)(\varepsilon+1)}.
\end{aligned}\right.
\end{equation}
As $\varepsilon\to 0$, one has 
\begin{equation}\nonumber
\left\{\begin{aligned}  
&a_+ = b_- \to   \frac{ -1}{e (e^2-1) }=:a_+^0 = b_-^0, \\ 
&a_0 =b_0 \to \frac{-e}{e^2-1}=:a_0^0 = b_0^0, \\ 
&a_- = b_+ \to  \frac{-e^3}{e^2-1}=:a_-^0 = b_+^0.
\end{aligned}\right.
\end{equation}
Define $p^0:=p^0_1e^{\ii x}$ with
\begin{equation}\nonumber
p^0_1  = \left\{\begin{aligned}
 & a_+^0 e^z + b_+^0 e^{-z} + c_+ ,\quad&& \text{ if }z\in (1,2),\\
 & a_0^0 e^z + b_0^0 e^{-z} + c_0 ,\quad&& \text{ if } z\in (-1,1),\\
 & a_-^0 e^z + b_-^0 e^{-z} + c_- ,\quad&& \text{ if } z\in (-2,-1).
\end{aligned}\right.
\end{equation}
It is straightforward to verify that $p^0$ solves \eqref{A-1} with $\varepsilon=0$. Finally, using the explicit expressions above, one obtains 
\begin{equation}\nonumber
\begin{aligned}
\|p^\varepsilon- p^0\|_{H^1(\Omega)}^2 \leq  &~ C(\|p-p_0\|_{L^2(-2,2)}^2+ \|p'-p_0'\|_{L^2(-2,2)}^2)\\
\leq&~ C \int_1^2  (a_+-a_+^0)^2e^{2z} + (b_+-b_+^0)^2e^{-2z} \dd z\\
&+ C \int_{-1}^1  (a_0-a_0^0)^2e^{2z} + (b_0-b_0^0)^2e^{-2z}\dd z\\
&+ C \int_{-2}^1  (a_- -a_-^0)^2e^{2z} + (b_- -b_-^0)^2e^{-2z}\dd z.
\end{aligned} 
\end{equation}
Since $|a_i-a_i^0|$ and $|b_i-b_i^0|$ converge linearly to zero as $\varepsilon\to0$ for $i=+,0,-$, we conclude that $p^\varepsilon$ converges to $p^0$ linearly in $H^1(\Omega)$ as $\varepsilon\to0$.

\section*{Acknowledgments}
This work is supported in part by NSFC 12271237. 









\medskip
Received xxxx 20xx; revised xxxx 20xx; early access xxxx 20xx.
\medskip

\end{document}